\documentclass[10pt]{amsart}

\usepackage{hyperref}

\usepackage{amssymb}
\usepackage{amsmath}
\let\newpf\proof \let\proof\relax 
\newenvironment{pf}{\newpf[\proofname]}{\qed\endtrivlist}

\def\DC{{\mathrm{DC}}}

\def\bm{\begin{matrix}}
\def\em{\end{matrix}}

\newcommand{\bt}{\begin{thm}}
\newcommand{\et}{\end{thm}}

\newcommand{\bl}{\begin{lemma}}
\newcommand{\el}{\end{lemma}}

\newcommand{\beq}{\begin{eqnarray}}
\newcommand{\eeq}{\end{eqnarray}}

\def\be{\begin{equation}}
\def\ee{\end{equation}}

\def\ba{{\begin{align}}}
\def\ea{{\end{align}}}

\def\SO{{\mathrm {SO}}}

\newcommand{\la}{\label}

\def\H{{\mathbb H}}

\def\0{{\mathbf 0}}

\def\SL{{\mathrm {SL}}}

\def\PSL{{\mathrm {PSL}}}

\newtheorem{thm}{Theorem}[section]
\newtheorem{cor}[thm]{Corollary}

\newtheorem{lemma}[thm]{Lemma}

\theoremstyle{remark}
\newtheorem{rem}{Remark}[section]

\newtheorem*{conj}{Spectral Dichotomy Conjecture}

\numberwithin{equation}{section}

\def \bn {\hfill \\ \smallskip\noindent}

\theoremstyle{definition}

\def\proof{\bn {\bf Proof.} }

\def\note#1
{\marginpar
{\tiny $\leftarrow$
\par
\hfuzz=20pt \hbadness=9000 \hyphenpenalty=-100 \exhyphenpenalty=-100
\pretolerance=-1 \tolerance=9999 \doublehyphendemerits=-100000
\finalhyphendemerits=-100000 \baselineskip=6pt
#1}\hfuzz=1pt}

\def\tr{{\text{tr}}}

\newcommand{\id}{\operatorname{id}}

\newcommand{\eps}{{\epsilon}}

\newcommand{\C}{{\mathbb C}}

\newcommand{\Q}{{\mathbb Q}}
\newcommand{\R}{{\mathbb R}}

\newcommand{\Z}{{\mathbb Z}}

\def\B0{{\bold{0}}}

\catcode`\@=12

\def\Empty{}
\newcommand\oplabel[1]{
  \def\OpArg{#1} \ifx \OpArg\Empty {} \else
  	\label{#1}
  \fi}
		
\newcommand{\comm}[1]{}
\newcommand{\comment}[1]{}

\begin{document}

\title[Continuity of spectral measures]
{H\"older continuity of absolutely continuous spectral measures for
one-frequency Schr\"odinger operators}

\author{Artur Avila and Svetlana Jitomirskaya}

\address{
CNRS UMR 7599,
Laboratoire de Probabilit\'es et Mod\`eles al\'eatoires\\
Universit\'e Pierre et Marie Curie--Bo\^\i te courrier 188\\
75252--Paris Cedex 05, France
}
\email{artur@ccr.jussieu.fr}

\address{
University of California, Irvine, California, 92697
}
\email{szhitomi@uci.edu}
\thanks{$^\dag$This work was supported in part by NSF, grant DMS-0601081,
and BSF, grant 2006483.  This research was partially conducted during the period A.A.
served as a Clay Research Fellow.}

\begin{abstract}
We establish sharp results on the modulus of continuity of the distribution
of the spectral measure for one-frequency Schr\"odinger operators with
Diophantine frequencies in the region of absolutely continuous spectrum.
More precisely, we establish $1/2$-H\"older continuity near almost reducible
energies (an essential support of absolutely continuous spectrum).
For non-perturbatively small potentials (and for the almost Mathieu operator
with subcritical coupling), our results apply for all energies.


\end{abstract}

\setcounter{tocdepth}{1}

\maketitle


\section{Introduction}

In this work we study absolutely continuous spectral measures of
(one-frequency) quasiperiodic Schr\"odinger operators $H=H_{\lambda v,\alpha,\theta}$
defined on $\ell^2(\Z)$
\be \label{11}
(H u)_n=u_{n+1}+u_{n-1}+\lambda v(\theta+n \alpha) u_n
\ee
where $v$ is the potential,
$\lambda \in \R$ is the coupling constant, $\alpha \in \R \setminus \Q$ is the
frequency and $\theta \in \R$ is the phase.  A central example is
given by the almost Mathieu operator, when $v(x)=2\cos(2 \pi x)$.

Except where otherwise noted, below we assume the frequency $\alpha$ to
be Diophantine in the usual sense (see definition in Section \ref{3}), and $v$ analytic.

Absolutely continuous spectrum occurs only rarely in one-dimensional Schr\"odinger operators \cite{reml}.
Until recently it was expected that in the class of ergodic Schr\"odinger operators it only occurs
for almost periodic potentials, a conjecture recently disproved \cite{akl}. Quasiperiodic operators
with analytic potential stand out in this respect as the family 
$\{H_{\lambda v,\alpha,\theta}\}_{\lambda \in \R},$ for small couplings $\lambda,$ is always in the
metallic phase (has good transport properties) with zero Lyapunov exponents and absolutely continuous spectrum.

We will be concerned with the regularity of spectral measures.  More
precisely, given a function
$f \in \ell^2(\Z)$ with $\|f\|=1$, and letting
$\mu^f=\mu^f_{\lambda v,\alpha,\theta}$ be
the associated spectral measure\footnote {That is, $\mu^f(X)=\|\Pi_X(f)\|^2$
where $\Pi_X:\ell^2(\Z) \to \ell^2(\Z)$ is the spectral projection
associated to the Borel set $X \subset \R$.}, what
can be said of the modulus of continuity of the distribution of $\mu^f$? 
We will assume that $f$ is a reasonably localized function in the sense
that $f \in \ell^1(\Z)$ (notice that without regularity assumptions, there
are no non-trivial restrictions on $\mu^f$: any probability measure
absolutely continuous with respect to some spectral measure is still a
spectral measure).  Our first result concerns small potentials:

\begin{thm} \label {small}

For every $v \in C^\omega(\R/\Z,\R)$, there exists
$\lambda_0=\lambda_0(v)>0$ such that if $|\lambda|<\lambda_0$ and $\alpha$
is Diophantine then $\mu^f_{\lambda v,\alpha,\theta}(J) \leq
C(\alpha,\lambda v) |J|^{1/2} \|f\|_{\ell^1}^2$, for all intervals $J$ and all $\theta.$  For the almost Mathieu
operator, one can take $\lambda_0=1$.

\end{thm}

\begin{rem}

The smallness constant $\lambda_0$ only depends on bounds on the
analytic extension of $v$ to some band $|\Im x|<\epsilon$.  This is
important for applications to arbitrary potentials (see below).
The constant $C$ depends on bounds on the analytic extension of $\lambda v$
and on the Diophantine properties of $\alpha$.

\end{rem}

Recall that averaging the distributions of spectral measures with respect to
the phase $\theta$ yields the integrated density of states (i.d.s.), whose
regularity is therefore significantly simpler to analyze.  Indeed in \cite
{AJ}, it is shown that the i.d.s. is $1/2$-H\"older (and no more, see below)
in the setting of Theorem \ref {small}.  The averaged $1/2$-H\"older
behavior is compatible with point spectrum
(consider the almost Mathieu operator with $\lambda>1$, \cite
{J}, \cite {AJ}), and hence
discontinuous distributions.  The key point of
Theorem \ref {small} is that here we are
able to control the behavior of each individual spectral measure, uniformly
on $\theta$.

The study of small potentials is not merely interesting on its own: it gives
information about the absolutely continuous spectrum of an arbitrary
potential.  To make this precise, one introduces the notion of {\it almost
reducibility}: roughly speaking
an energy is almost reducible if the associated cocycle
$(\alpha,A^{(E-\lambda v)})$
(a dynamical system
\be
(x,w) \mapsto (x+\alpha,A^{(E-\lambda v)} \cdot w),
\ee
\be
A^{(E-\lambda v)}=\left (\bm E-\lambda v(x) & -1 \\ 1 & 0 \em \right ),
\ee
that describes the behavior of solutions of
the eigenvalue equation
$H_{\lambda v,\alpha,\theta} u=E u$), is analytically conjugate (in a
uniform band) to the associated cocycle of some $(\alpha,A^{(E'-v')})$
with $v'$ arbitrarily small.  It follows from renormalization \cite {AK},
\cite {AK2}, that almost reducible energies (indeed {\it reducible}
energies, for which $v'$ can be taken as $0$) form an essential support of
absolutely continuous spectrum.  In \cite {AJ}, almost reducibility was
proved for all energies in the case of small potentials (indeed the
same setting of Theorem \ref {small}), which implies that almost
reducibility is {\it stable}
(in particular, the set of almost reducible energies is open).

\begin{thm} \label {gen}

Let $v \in C^\omega(\R/\Z,\R)$ and let $\alpha$ be Diophantine.  Then for
any almost reducible energy $E \in \Sigma_{v,\alpha}$, (thus for a.e. energy in $\Sigma_{v,\alpha}^{ac}$) 
there exists $C,\epsilon>0$ such that if $J \subset (E-\epsilon,E+\epsilon)$
is an interval then, for all $\theta,\;$ 
$\mu^f_{v,\alpha,\theta}(J) \leq C |J|^{1/2} \|f\|_{\ell^1}$.

\end{thm}

As far as we know Theorems \ref{gen},\ref{small} are the first results on fine properties of individual absolutely continuous spectral measures of ergodic operators. 
 
Let us call
attention to the following conjecture that clarifies the fundamental
importance of understanding almost reducibility:

\begin{conj}

For typical $v,\alpha,\theta$, $H_{v,\alpha,\theta}$
is the direct sum of operators $H_+$
and $H_-$ with disjoint spectra such that $H_+$ is ``localized'' and
$H_-$ is ``almost reducible''.

\end{conj}

\begin{rem}

\begin{enumerate}

\item Typical should be understood in the measure-theoretical sense of {\it
prevalence}.  In particular frequencies may be assumed to be Diophantine.

\item Localization for $H_+$
means both what is usually understood as Anderson
localization (pure point spectrum with exponentially decaying
eigenfunctions) or dynamical localization.\marginpar {Explain}
Almost reducibility for $H_-$
just means that the spectrum of $H_-$ is the closure of almost
reducible energies for $H$, but as described above, it indeed provides a
very fine spectral description: in particular the results
of \cite {AJ} and this paper apply to $H_-$, e.g., it has
absolutely continuous
spectral measures with $1/2$-H\"older distributions.

\item By Kotani theory, see e.g. \cite {LS}, if the conjecture holds then
$H_+$/$H_-$ must be defined by spectral projection on the
parts of the spectrum where the associated cocycle has positive/zero
Lyapunov exponent
\be
L(E)=\lim \frac {1} {n} \int_{\R/\Z} \ln \|A^{(E-v)}_n(x)\| dx,
\ee
\be
A_n^{(E-v)}(x)=A^{(E-v)}(x+(n-1) \alpha) \cdots A^{(E-v)}(x).
\ee
The result of disjointness of the spectra for this
decomposition was recently established (in the typical setting) \cite {A2}, \cite {A3}.

\item With $H_+$ defined as above, the precise spectral and dynamical description, particularly
dynamical localization, follows (in the typical setting)
from a minoration of the Lyapunov exponent through the spectrum of $H_+$,
using \cite {BG} and \cite {BJ3}.
  Such minoration is a consequence of disjointness of spectra \cite{A2,A3} and
continuity of the Lyapunov exponent \cite {BJ1}.  More is
known in this regime \cite{JL2},\cite {GS2}, \cite {GS3}.

\item What is still
incomplete in the above picture is the description of $H_-$. 
Zero Lyapunov exponent does not necessarily imply almost reducibility
(consider the critical almost Mathieu operator).
In \cite {A2}, \cite {A3}, it is shown that (in the typical setting)
energies in the spectrum of of $H_-$ satisfy not only $L(E)=0$ but the
stronger condition (called {\it subcriticality})
\be
\ln \|A_n(x)\|=o(n)
\ee
uniformly in some band $|\Im x|<\epsilon$.  The Spectral Dichotomy
Conjecture is thus reduced to the Almost Reducibility Conjecture
(the main outstanding problem
in the theory): subcriticality implies almost reducibility.

\end{enumerate}

\end{rem}

\subsection{Further perspective}

One should distinguish between two possible regimes of small $|\lambda|$
(similar considerations can be applied to the analysis of large coupling).
One is {\it perturbative}, meaning that the smallness condition on
$|\lambda|$ depends not only on the potential $v$, but also on the frequency
$\alpha$: the key resulting limitation is that the analysis at a given
coupling, however small, has to exclude a positive Lebesgue measure set of
$\alpha$. Such exclusions are inherent to the KAM-type methods that have been traditionally used in this context.
The other, stronger regime, is called {\it non-perturbative}, meaning that the
smallness condition on $|\lambda|$ only depends on the potential, leading to results that hold for
almost every $\alpha.$ 

A thorough study of absolutely continuous spectrum 
 of operators \eqref{11} 
in the case of small analytic potentials in the perturbative regime
 was done by Eliasson \cite{E}. 
He proved the reducibility of the associated cocycle for almost all
energies in the spectrum and fine estimates on solutions for the other
energies, by
developing  a sophisticated KAM scheme, which avoided
the limitations of earlier KAM methods (that go back to the work of
Dinaburg-Sinai \cite{ds} and that excluded parts of the
spectrum from consideration). This allowed him in particular to
conclude purely absolutely continuous spectrum.

A thorough study of absolutely continuous spectrum 
 of operators \eqref{11}  in the non-perturbative regime of \cite{BJ2}
 was done in \cite{AJ} where we used some techniques of \cite{BJ2} to
 obtain localization estimates for all energies for the dual model,
 and developed quantitative Aubry duality theory, which allowed us, in
 particular, to
 conclude almost reducibility for all energies (including those for
 which neither dual localization nor reducibility hold). The smallness
 condition on the coupling constant in \cite{AJ} coincides with that
 of \cite{BJ2}. In particular, for the almost Mathieu operator, all
 the estimates and conclusions hold throughout the subcritical regime
 $\lambda<1.$ 

The analyses of \cite{E} and \cite{AJ} allowed to obtain sharp bounds
(H\"older-1/2 continuity) for the integrated density of states, for
Diophantine frequencies.
This was done, in perturbative and
non-perturbative regimes in correspondingly \cite{Am} and \cite{AJ}. 

Earlier, Goldstein-Schlag \cite {GS2} had  shown H\"older continuity
of the integrated density of states for a full Lebesgue  measure subset of Diophantine frequencies in the regime of positive Lyapunov
exponents, with the result becoming almost sharp for the super-critical almost
Mathieu operator:
$(1/2-\epsilon)$-H\"older for any $\epsilon,$ and $|\lambda|>1$ 
. For this model
their result also gives the same bound in the sub-critical regime
$|\lambda|<1,$ by duality. Before that Bourgain \cite {B1} had obtained almost
$1/2$-H\"older
continuity for almost Mathieu type potentials in the perturbative regime, for Diophantine $\alpha$ and
$\ln |\lambda|$ large (depending on $\alpha$). 

There were no results however, neither recently nor previously, on the
modulus of continuity of the individual spectral measures, even in the
perturbative regime. In this paper we achieve this by applying methods developed in
\cite{AJ}  combined with a dynamical reformulation of the power-law subordinacy
techniques of \cite{JL1},\cite{JL2}.  As mentioned above, our all energy results hold
throughout the regime of \cite{BJ2}, and in particular, for all
sub-critical almost-Mathieu operators. The general absolutely continuous case is obtained through 
a reduction to the small potential case and almost reducibility result of \cite{AK}.

Our estimate is optimal in several ways. First, there are square-root
singularities at the boundaries of gaps (e.g., \cite{puig2}), so the modulus of continuity
cannot be improved. Also, since the integrated
density of states satisfies
\be \label{N} 
N(E)=\lim_{n \to \infty} \frac {1} {n} \sum_{k=0}^{n-1}
\mu^{\sigma^k(f)}(0,E]
\ee
($\sigma:\ell^2(\Z) \to \ell^2(\Z)$ denotes the shift),
the spectral measures of $\ell^1$ functions
cannot have higher modulus of continuity
than $N(E)$.  There are examples with lower regularity of $N(E)$
that demonstrate that Diophantine condition on $\alpha$ as well as a
condition on $\lambda$ are essential here.  In particular, it is known
that for the almost Mathieu operator for a certain non-empty set of $\alpha$ which
satisfy good Diophantine properties (but has zero Lebesgue measure)
and $\lambda=1$, the integrated density of states is not H\"older
(\cite {B2}, Remark after Corollary 8.6).
Additionally, for any $\lambda \neq 0$ and generic $\alpha$, the integrated
density of states is not H\"older (this is because the Lyapunov exponent is
discontinuous at rational $\alpha$, which easily
implies that it is not H\"older for generic $\alpha$. Such
discontinuity holds for the almost Mathieu operator and presumably generically).

\begin{rem} \label {generalizations}

As our approach is non-perturbative and non-KAM, it is not expected to
break down at the Brjuno condition and can potentially be extended
much further. While, as mentioned above, the exact modulus of continuity
should depend on the Diophantine properties
for very well approximated $\alpha$ we expect the same methods
to work for small rate of exponential approximation as well.
We do not pursue it here though.

\end{rem}


\section{Preliminaries}

For a bounded analytic (perhaps matrix valued) function $f$
defined on a strip $\{|\Im z|<\epsilon\}$ and extending continuously to the
boundary, we
let $\|f\|_\epsilon=\sup_{|\Im z|<\epsilon} |f(z)|$.  If $f$ is a bounded
continuous function on $\R$, we let $\|f\|_0=\sup_{x \in \R} |f(x)|$.

\subsection{Cocycles}

Let $\alpha \in \R \setminus \Q$, $A \in C^0(\R/\Z,\SL(2,\C))$.  We call
$(\alpha,A)$ a {\it (complex) cocycle}.
The {\it Lyapunov exponent} is given by the formula
\be
L(\alpha,A)=\lim_{n \to \infty} \frac {1} {n} \int \ln \|A_n(x)\| dx,
\ee
where $A_n$, $n \in \Z$, is defined by $(\alpha,A)^n=(n \alpha,A_n)$, so
that for $n \geq 0$,
\be
A_n(x)=A(x+(n-1)\alpha) \cdots A(x).
\ee

We say that $(\alpha,A)$ is {\it uniformly hyperbolic} if there exists a
continuous splitting $\C^2=E^s(x) \oplus E^u(x)$, $x \in \R/\Z$ such that
for some $C>0$, $c>0$, and for every $n \geq 0$,
$\|A_n(x) \cdot w\| \leq C e^{-c n} \|w\|$, $w \in E^s(x)$
and $\|A_{-n}(x) \cdot w\| \leq C e^{-cn} \|w\|$, $w \in E^u(x)$.  In this
case, of course $L(\alpha,A)>0$.

Given two cocycles $(\alpha,A^{(1)})$ and $(\alpha,A^{(2)})$,
a {\it (complex) conjugacy}
between them is a continuous $B:\R/\Z \to \SL(2,\C)$ such that
\be \label {conj}   
A^{(2)}(x)=B(x+\alpha)A^{(1)}(x) B(x)^{-1}.
\ee


We assume now that $(\alpha,A)$ is a {\it real} cocycle, that is,
$A \in C^0(\R/\Z,\SL(2,\R))$.  The notion of real conjugacy (between real
cocycles) is the same as before, except that we ask for $B \in
C^0(\R/\Z,\PSL(2,\R))$.  Real conjugacies still preserve the Lyapunov
exponent.  

We say that a real cocycle $(\alpha,A)$ is (analytically)
{\it reducible} if it is (real) conjugate to a
constant cocycle, and the conjugacy is analytic.  We say that it is {\it
almost reducible} if there exists a sequence $A^{(n)} \in
C^\omega(\R/\Z,\R)$ converging (uniformly in some band $\{|\Im
z|<\epsilon\}$) to a constant, such that $(\alpha,A^{(n)})$ is conjugated to
$(\alpha,A)$, and the conjugacies extend holomorphically to some fixed band:
$B^{(n)} \in C^\omega_\epsilon(\R/\Z,\PSL(2,\R))$.\footnote {In fact
this last property is automatic, but this is non-trivial (it follows
from the openness of almost reducibility).}

\subsection{Schr\"odinger operators}

We consider now Schr\"odinger operators $\{H_{v,\alpha,\theta}\}_{\theta \in
\R}$ (we incorporate the coupling constant into $v$).  The spectrum
$\Sigma=\Sigma_{v,\alpha}$ does not depend on $\theta$,
and it is the set of $E$ such that $(\alpha,A^{(E-v)})$ is not uniformly
hyperbolic, with $A^{(E-v)}$ as in the introduction.

For $f \in l^2(\Z)$ the spectral measure $\mu=\mu^f_x$  is defined so that
\be \label {resolvent}
\langle (H_x-E)^{-1} f,f \rangle=\int_\R \frac {1} {E'-E} d\mu(E')
\ee
holds for $E$ in the resolvent set $\C \setminus \Sigma$.  
Alternatively,
for a Borel set $X$,
\be \label {pix}
\mu^f_{v,\alpha,\theta}(X)=\|\Pi_X f\|^2,
\ee
where $\Pi_X$ is the corresponding
spectral projection.

The integrated density of states  is the function $N:\R \to [0,1]$
that can be defined by (\ref{N})
\comm{\be
N(E)=\int_{\R/\Z} \mu^f_{v,\alpha,\theta}(-\infty,E] d\theta,
\ee
where $f \in l^2(\Z)$ is such that $\|f\|_{l^2(\Z)}=1$ (the definition is
independent of the choice of $f$).}
It is a continuous non-decreasing surjective (for bounded potentials) function.
The Thouless formula relates the Lyapunov exponent to
the integrated density of states
\be
L(E)=\int_\R \ln |E'-E| dN(E').
\ee

\subsection{Almost reducibility and the support of absolutely continuous
spectrum}

We justify the claim made in the introduction that almost reducible energies
support the absolutely continuous part of the spectral measures.

By \cite {LS}, the set $\Sigma_0=\{L(E)=0\}$ (which is closed by \cite {BJ1})
is the essential support of ac
spectrum, so that the ac spectral measures are precisely those
probability measures on $\Sigma_0$ which are equivalent to Lebesgue.

\begin{thm}[\cite {AFK}]

If $\alpha \in \R \setminus \Q$ then for almost every
$E \in \Sigma_0$, $(\alpha,A^{(E-v)})$ is real
analytically conjugated to a cocycle of rotations, i.e., taking values in
$\SO(2,\R)$.

\end{thm}

This result was proved in \cite {AK} under a full measure condition on
$\alpha$ (which is stronger than Diophantine).  For Diophantine $\alpha$, it
can also be obtained as a consequence of \cite {AJ} and \cite {AK2}.

It is easy to see that for any $\alpha \in \R \setminus \Q$,
analytic cocycles of rotations are almost reducible.
Moreover, if $\alpha$ is
Diophantine (more generally, if the best rational approximations to $\alpha$
are subexponential), then analytic cocycles of rotations are reducible.

\subsection{Almost reducibility in Schr\"odinger form}

While almost reducibility allows one to conjugate the dynamics of the
cocycle close to a constant, it is rather convenient to have the conjugated
cocycle in Schr\"odinger form, since many results (particularly the ones
depending on Aubry duality, as the ones obtained in \cite {AJ}) are obtained
only in this setting.  The following result takes care of this.

\begin{lemma} \label {almost}

Let $(\alpha,A) \in \R \setminus \Q \times C^\omega(\R/\Z,\SL(2,\R))$
be almost reducible.  Then there exists
$\epsilon_0>0$ such that for every $\gamma>0$, there exists
$v \in C^\omega_{\epsilon_0}(\R/\Z,\R)$ with $\|v\|_{\epsilon_0}<\gamma$, $E
\in \R$
and $B \in C^\omega_{\epsilon_0}(\R/\Z,\PSL(2,\R))$ such that $B(x+\alpha)
A(x) B(x)^{-1}=A^{(E-v)}(x)$.  Moreover, for
every $0<\epsilon \leq \epsilon_0$, there exists
$\delta>0$ such that if $\tilde
A \in C^\omega(\R/\Z,\SL(2,\R))$ is such that $\|\tilde
A-A\|_\epsilon<\delta$ then there exists $\tilde v \in
C^\omega(\R/\Z,\R)$, such that $\|\tilde v\|_\epsilon<\gamma$ and $\tilde B \in
C^\omega_\epsilon(\R/\Z,\PSL(2,\R))$ such that $\|\tilde B-B\|<\gamma$ and
$\tilde B(x+\alpha) \tilde A(x)
\tilde B(x)^{-1}=A^{(E-\tilde v)}(x)$.

\end{lemma}

For the proof, one basically just needs to be able to convert
non-Schr\"odinger perturbations of Schr\"odinger cocycles
to Schr\"odinger form.  This problem is
studied in \cite {A3}.
For completeness, we will give a much simpler (unpublished) argument of
Avila-Krikorian which is enough for our purposes.

\begin{lemma} \label {sch}

Let $v \in C^\omega_\epsilon(\R/\Z,\R)$ and
$\alpha \in \R \setminus \Q$, be such that $1/v \in
C^\omega_\epsilon(\R/\Z,\R)$.
If $A \in C^\omega_\epsilon(\R/\Z,\SL(2,\R))$ and
$\|A-A^{(v)}\|_\epsilon$ is sufficiently small (depending on
$\|v\|_\epsilon$ and $\|1/v\|_\epsilon$),
then there exists $v' \in C^\omega_\epsilon(\R/\Z,\R)$ and
$B \in C^\omega_\epsilon(\R/\Z,\SL(2,\R))$
such that $\|v'-v\|_\epsilon$ and
$\|B-\id\|_\epsilon$ are
small and
$B(x+\alpha) A(x) B(x)^{-1}=A^{(v')}(x)$.

\end{lemma}

\begin{pf}

Let $w=\left (\bm w_1 & w_2 \\ w_3 &
-w_1 \em \right ) \in C^\omega_\epsilon(\R/\Z,\mathrm {sl}(2,\R))$ be such that
$\|w\|_\epsilon$ is small and
$A=A^{(v)} e^w$.  Let $s=\left (\bm s_1 & s_2 \\ s_3 & -s_1 \em
\right ) \in C^\omega_\epsilon(\R/\Z,\mathrm {sl}(2,\R))$ be defined by $s_1=0$,
$s_2(x)=w_2(x)+\frac {w_1(x)} {v(x)}$, $s_3(x)=-\frac {w_1(x-\alpha)}
{v(x-\alpha)}$ and let $\tilde v \in C^\omega_\epsilon(\R/\Z,\R)$ be given
by
$$
\tilde v(x)=v(x)-w_3(x)+w_2(x+\alpha)+\frac {w_1(x+\alpha)} {v(x+\alpha)}+
v(x) w_1(x)-\frac {w_1(x-\alpha)} {v(x-\alpha)}.
$$
Then $\|\tilde v-v\|_\epsilon \leq C \|w\|_\epsilon$ and
$e^{s(x+\alpha)} A(x) e^{-s(x)}$ is of the form $A^{(\tilde
v)} e^{\tilde w}$ where $\|\tilde w\|_\epsilon \leq C \|w\|_\epsilon^2$,
for some constant $C$ depending on $\|v\|_\epsilon$ and $\|1/v\|_\epsilon$. 
The result follows by iteration.
\end{pf}

  \begin{rem}
\begin{enumerate}
\item This result with only assuming $v \in
  C^\omega_\epsilon(\R/\Z,\R)$ to be non-identically zero is proved in
  \cite{A3}.
\item For Diophantine $\alpha$ the result holds with no conditions on
  $v\in
  C^\omega_\epsilon(\R/\Z,\R).$
\end{enumerate}

\end{rem}
  
\comm{ We will use (a very simple instance) of the
following result.

\begin{lemma}[\cite {A3}] \label {sch}

Let $v \in C^\omega_\epsilon(\R/\Z,\R)$ be non-identically zero.
If $A \in C^\omega_\epsilon(\R/\Z,\SL(2,\R))$ and
$\|A-A^{(v)}\|_\epsilon$ is sufficiently small,
then there exists $v' \in C^\omega_\epsilon(\R/\Z,\R)$ and
$B \in C^\omega_\epsilon(\R/\Z,\SL(2,\R))$
such that $\|v'-v\|_\epsilon$ and
$\|B-\id\|_\epsilon$ are
small and
$B(x+\alpha) A(x) B(x)^{-1}=A^{(v')}(x)$.\footnote {In our application,
$\frac {1} {v} \in C^\omega_\epsilon(\R/\Z,\R)$, and the lemma is much
simpler to prove.}

\end{lemma}}

\noindent {\it Proof of Lemma \ref {almost}.}
Since $(\alpha,A)$ is almost reducible, if $\epsilon_0>0$ is small then
there exists a sequence $B^{(n)} \in
C^\omega_{\epsilon_0}(\R/\Z,\PSL(2,\R))$
and $A_* \in \SL(2,\R)$ such that $\|B^{(n)}(x+\alpha) A(x)
B^{(n)}(x)^{-1}-A_*\|_{\epsilon_0} \to 0$.  Let us show that, up to
changing $B^{(n)}$ to $C^{(n)} B^{(n)}$ for an appropriate choice of
$C^{(n)} \in C^\omega_{\epsilon_0}(\R/\Z,\SL(2,\R))$,
we may assume that $A_*$ is of
the form $\left (\bm E & -1 \\ 1 & 0 \em \right )$ with $E \neq 0$.
Indeed:
\begin{enumerate}
\item If $|\tr A_*|>2$, by converting first to the diagonal form, we
  can find $\tilde C_{A_*} \in \SL(2,\R)$ such
that $\tilde C_{A_*} A_* \tilde C^{-1}_{A_*}=\left (\bm \tr A_* & -1 \\ 1 & 0 \em \right
)$, so we can just take $C^{(n)}=C_{A_*}$.
\item If $|\tr A_*|<2$,
there exists $\tilde C_{A_*} \in \SL(2,\R)$
such that $\tilde C_{A_*} A_* \tilde
C_{A_*}^{-1}=R_\theta$ for some $\theta \not= k/2, k\in \Z$. 
If $0<\sin 2 \pi \theta<1$, then let
$$
C_\theta^{-1}=
\frac {1} {(\sin 2 \pi \theta)^{1/2}}
\left (\bm 0&-\sin 2 \pi \theta \\ 1 & -\cos 2 \pi \theta
\em \right ),
$$
so that $C_\theta R_\theta C_\theta^{-1}=
\left (\bm 2 \cos 2 \pi \theta & -1 \\ 1 & 0 \em
\right )$, and we can take $C^{(n)}=C_\theta \tilde C_{A_*}$.  Otherwise, let
$k \in \Z$ be such that
$0<\sin 2 \pi (\theta+k\alpha)<1$, and take
$C^{(n)}(x)=C_{\theta+k \alpha}
R_{k x} \tilde C_{A_*}$.
\item If $|\tr A_*|=2$,
there exists $\tilde C^{(n)} \in \SL(2,\R)$
such that $\tilde C^{(n)} A_* (\tilde C^{(n)})^{-1} \to R_\theta$, where
$\theta=0$ or $\theta=1/2.$ Indeed, either $A_*$ is equal to such
$R_\theta$ or we can assume it is in the Jordan form, in which case
one can take 
$$ \tilde C^{(n)} =\left (\bm \epsilon_n&0 \\ \epsilon_n & \frac 1\epsilon_n
\em \right ),
$$
By choosing  $\epsilon_n$ appropriately, we may also assume that $\|\tilde C^{(n)}\|^2
\|B^{(n)}(x+\alpha) A(x) B^{(n)}(x)^{-1}-A_*\|_{\epsilon_0} \to 0$.  Choosing again
$k \in \Z$ such that $0<\sin 2 \pi
(\theta+k \alpha)<1$ and can take $C^{(n)}=C_{\theta+k \alpha} R_{k x}
\tilde C^{(n)}$.
\end{enumerate}

 Now the first statement follows from
Lemma \ref {sch}.  For the second statement,
apply again Lemma \ref {sch}.
\qed

\section{Estimates on the dynamics}\la{3}

Here we describe the \cite {AJ}
estimates on the dynamics of almost reducible cocycles.

\subsection{Rational approximations}

Let $q_n$ be the denominators of the
approximants of $\alpha$.  We recall the basic properties:
\be \label {b1}
\|q_n \alpha\|_{\R/\Z}=\inf_{1 \leq k \leq q_{n+1}-1}
\|k\alpha\|_{\R/\Z},
\ee
\be \label {b2}
1 \geq q_{n+1} \|q_n \alpha\|_{\R/\Z} \geq 1/2.
\ee

We say that $\alpha$ is Diophantine if $\frac {\ln q_{n+1}} {\ln
q_n}=O(1)$.  Let $\DC \subset \R$ be the set of Diophantine numbers.

\subsection{Resonances}

Let $\alpha \in \R$, $\theta \in \R$, $\epsilon_0>0$.
We say that $k$ is an $\epsilon_0$-resonance if $$\|2
\theta-k\alpha\|_{\R/\Z} \leq e^{-|k|\epsilon_0}$$ and $$\|2
\theta-k\alpha\|_{\R/\Z}=\min_{|j| \leq |k|} \|2 \theta-j\alpha\|_{\R/\Z}.$$

\begin{rem}

In particular, there always exists at least one resonance, $0$.  If $\alpha
\in \DC(\kappa,\tau)$, $\|2
\theta-k\alpha\|_{\R/\Z} \leq e^{-|k|\epsilon_0}$ implies $$\|2
\theta-k\alpha\|_{\R/\Z}=\min_{|j| \leq |k|} \|2 \theta-j\alpha\|_{\R/\Z}$$
for $k>C(\kappa,\tau)$.

\end{rem}

For fixed $\alpha$ and $\theta$, we order the
$\epsilon_0$-resonances $0=n_0<|n_1| \leq |n_2| \leq ...$.  We say that $\theta$
is $\epsilon_0$-resonant if the set of resonances is infinite.
If $\theta$ is non-resonant, with the set of resonances
$\{n_0,\ldots,n_j\}$ we formally set $n_{j+1}=\infty.$  The
Diophantine condition immediately implies exponential repulsion of
resonances:

\begin{lemma} \label {following resonance}

If $\alpha \in \DC$, then
$$|n_{j+1}| \geq c \|2 \theta-n_j \alpha\|_{\R/\Z}^{-c} \geq
c e^{c \epsilon_0 |n_j|},$$ where $c=c(\alpha,\epsilon_0)>0$.

\end{lemma}

\subsection{Dynamical estimates}

Let us say that a cocycle $(\alpha,A) \in \R \setminus \Q \times
C^\omega(\R/\Z,\SL(2,\R))$ is $(C,c,\epsilon_0)$-good if
there exists $\theta \in \R$ with the following property: for any finite
$\epsilon_0$-resonance $n_j$ associated
to $\alpha$ and $\theta$, denoting $n=|n_j|+1$ and
$N=|n_{j+1}|$, there exists $\Phi:\R/\Z \to \SL(2,\C)$
analytic with $\|\Phi\|_{c n^{-C}} \leq C n^C$ such that
\be
\Phi(x+\alpha)A(x)\Phi(x)^{-1}=\left (\bm e^{2 \pi i \theta} & 0\\0 & e^{-2
\pi i \theta} \em \right )+\left (\bm q_1(x)&q(x)\\
q_3(x)&q_4(x) \em \right ),\ee
with
\be \label {q1}
\|q_1\|_{c n^{-C}},\|q_3\|_{c n^{-C}},\|q_4\|_{c n^{-C}}
\leq C e^{-c N}
\ee
and
\be \label {q}
\|q\|_{c n^{-C}} \leq C e^{-cn (\ln (1+n))^{-C}}.
\ee

The following is one of the main estimates of \cite {AJ} (combining Theorems
3.3, 3.4 and 5.1 of \cite {AJ}):

\begin{thm} \label {Cc} [see Theorems 3.4 and 5.1 of \cite {AJ}]

There exists a constant $c_0>0$ with the following property.
Let $v \in C^\omega(\R/\Z,\R)$ and $E \in
\Sigma_{v,\alpha}$, $(\alpha,A)$.
If for some $0<\epsilon<1$, $\|v\|_\epsilon<c_0 \epsilon^3$ then
$(\alpha,A)$ is $(C,c,\epsilon_0)$-good
for some constants $c=c(\epsilon,\alpha)>0$,
$C=C(\epsilon,\alpha)>0$ and $\epsilon_0=\epsilon_0(\epsilon)$.

\end{thm}

A more precise result is available for the almost Mathieu operator (still a
combination of Theorems 3.4 and 5.1 of \cite {AJ}):

\begin{thm} \label {CC2} [see Theorems 3.4 and 5.1 of \cite {AJ}]

For every $0<\lambda_0<1$ and $\alpha$ Diophantine, there exists
$C=C(\lambda_0,\alpha),c=c(\lambda_0,\alpha),
\epsilon_0=\epsilon_0(\lambda_0)>0$
such that for $v=2 \lambda \cos 2 \pi (x+\theta)$ with
$|\lambda|<\lambda_0$, and $E \in \Sigma_{v,\alpha}$, $(\alpha,A^{(E-v)})$
is $(C,c,\epsilon_0)$-good.

\end{thm}

Coupling Theorem \ref {Cc} and Lemma \ref {almost} we immediately get:

\begin{thm} \label {CC3}

Let $\alpha$ be Diophantine and let $A \in C^\omega(\R/\Z,\SL(2,\R))$.
If $(\alpha,A)$ is almost reducible then there exists $\bar{\epsilon}>0$ such that for every $0<\epsilon<\bar{\epsilon}$ there
exist $\delta,C,c,\epsilon_0>0$ such that if
$\tilde A \in C^\omega(\R/\Z,\SL(2,\R))$ is such that
$\|\tilde A-A\|_\epsilon<\delta$ and $(\alpha,\tilde A)$ is not uniformly
hyperbolic then $(\alpha,\tilde A)$ is $(C,c,\epsilon_0)$-good.

\end{thm}

\begin{rem}

Using \cite {A1}, Theorem 3.8, one can consider a stronger definition of
goodness, so that Theorem \ref {Cc}, and hence Theorem \ref {CC3},
and Theorem \ref {CC2}, still hold:
$\|\Phi\|_c \leq C n^C$, $\|q_j\|_c \leq C e^{-c N}$, $j=1,3,4$, and
$\|q\|_c \leq C e^{-c n}$.

\end{rem}

An immediate consequence of $(C,c,\epsilon_0)$-goodness
is (see \cite {AJ} for the easy argument):

\begin{lemma} \la{linear}

If $(\alpha,A)$ is $(C,c,\epsilon_0)$-good then for every $s \geq 0$ we have
$\|A_s\|_0 \leq C'(C,c,\epsilon_0,\alpha) (1+s)$.

\end{lemma}

\section{Regularity of the spectral measures at good energies}
\label {spectral measures cont}

Let $v \in C^\omega(\R/\Z,\R)$, $E \in \Sigma_{v,\alpha}$.  Let
$\mu_x=\mu^{e_{-1}}_{v,\alpha,x}+\mu^{e_0}_{v,\alpha,x}$ and
$e_i$ is the Dirac mass at $i \in \Z$.

Our main estimate is:

\begin{thm} \label {est}

If $(\alpha,A^{(E-v)})$ is $(C_0,c_0,\epsilon_0)$-good then for every $0<\epsilon<1$,
$\mu_x(E-\epsilon,E+\epsilon) \leq
C'(C_0,c_0,\epsilon_0,\alpha) \epsilon^{1/2}$.

\end{thm}

{\it Proof of Theorems \ref {small} and \ref {gen}.}
We first prove Theorem \ref {gen}.
By Theorem \ref {CC3}, $(\alpha,A^{(E'-v)})$ is $(C_0,c_0,\epsilon_0)$-good for any $E'$
near $E$ which is in the spectrum.  By Theorem \ref {est}, we get
\be \label {C'}
\mu_x(J) \leq C' |J|^{1/2}
\ee
for any interval containing such an $E'$, and
hence (since $\mu_x$ is supported on the spectrum), for any interval
contained in a neighborhood of $E$.
Let $\sigma:l^2(\Z) \to l^2(\Z)$ be the shift $f(i+1)=\sigma f(i)$.  Then
$\sigma H_{v,\alpha,x} \sigma^{-1}=H_{v,\alpha,x+\alpha}$.  Thus
$\mu^{\sigma f}_{x+\alpha}=\mu^f_x$ and
$\mu^{e_k}_x=\mu^{e_0}_{x+k\alpha} \leq \mu_{x+k\alpha}$.  By (\ref {pix}),
$\mu^f_x(E-\epsilon,E+\epsilon)^{1/2}$ defines a semi-norm on $l^2(\Z)$.
Therefore, by the triangle inequality,
$\mu^f_x(J)^{1/2} \leq \sum_{k \in \Z} |f(k)| (\mu_{x+k\alpha}(J))^{1/2}$,
and the result follows immediately from (\ref {C'}).

Theorem \ref {small} is proved analogously, using
Theorems \ref {Cc} and \ref {CC2} to establish
appropriate $(C,c)$-goodness.
\qed

It therefore remains to prove Theorem \ref{est} which we do in Section \ref{4.1}.

Through the end of this section, $A=A^{(E-v)}$.  We will use $C$ and $c$ for
large and small constants that only depend on $C_0$, $c_0,\epsilon_0,$ and $\alpha$.

\subsection{Spectral measures and $m$-functions}

In the study of $\mu=\mu_x$, we will use a
result of \cite {JL1} (or its improvement in \cite{kkl}),
interpreted in terms of cocycles.  In the definition
of the $m$-functions below, we follow the notation of \cite {JL2}.

We will consider energies $E+i\epsilon$, $E \in \R$, $\epsilon>0$.
Then there are non-zero solutions $u^\pm$
of $H u^\pm=(E+i\epsilon) u^\pm$ which
are $l^2$ at $\pm \infty$, well defined up to normalization.  We define
\be
m^\pm=\mp \frac {u^\pm_1} {u^\pm_0}.
\ee

It coincides with the Weyl-Titchmarsh $m$-function which is the Borel
transform of the spectral measure $\mu^\pm=\mu^\pm_{e_0}$ of the corresponding
half-line problem
with Dirichlet boundary conditions:
$$m^\pm(z)=\int\frac{d\mu^\pm(x)}{x-z},$$ (e.g.\cite{Cy}).

Thus $m^\pm$ has positive
imaginary part for every $\epsilon>0$.

Let
\be \label{M}
M(E+i\epsilon)=\int \frac {1} {E'-(E+i\epsilon)} d\mu(E').
\ee

Notice that $M(E+i\epsilon) \in \H=\{z,\, \Im z>0\}$.
We have 
\be\la{imm}
\Im M(E+i\epsilon)\geq \frac 1{2\epsilon} \mu(E-\epsilon,E+\epsilon).\ee

Then, as discussed in \cite {JL2},
\be
M=\frac {m^+ m^--1} {m^++m^-}.
\ee

As in \cite {JL2}, we define $m^+_\beta=R_{-\beta/2\pi} \cdot m^+,$
or, more generally, $$z_\beta =R_{-\beta/2\pi}z .$$  Those are Borel
transforms of the half-line spectral measures $\mu^\beta=\mu_{e_0}^\beta$ of
operator $H$ on $l^2([0,\infty))$ with boundary conditions $u_0\cos\beta+u_1\sin\beta=0.$ Here we
make use of the action of $\SL(2,\C)$ on $\overline \C$, $$\left (\bm
a&b\\c&d \em \right ) \cdot z=\frac {az+b} {cz+d}.$$ Let $\psi(z)=
\sup_\beta |z_\beta|.$  We have
\be \label {psi1}
\psi(z)^{-1} \leq \Im z \leq |z| \leq \psi(z).
\ee
where the first inequality easily follows from the invariance of
$\phi$, see below.
It was shown in \cite{dkl} that, as a 
corollary of the maximal modulus principle, one obtains
\be \la{Mpsi}
|M| \leq \psi(m^+).
\ee
This also can be shown directly by the following computation, that gives some more quantitative estimates.
Let $$\phi(z)=\frac {1+|z|^2} {2 \Im z}.$$  If $z \in \H$ then $\phi(z) \geq
1$. $\phi(z)$ is invariant with respect to the action of $R_{\beta}.$  Thus
the maximum of $|z_\beta|$ is attained when $z_\beta$ is purely imaginary
with $\Im z_\beta >1$ and is  easily checked to be equal to
$\phi(z)+(\phi(z)^2-1)^{1/2}$. Thus $$\psi(z)=\phi(z)+(\phi(z)^2-1)^{1/2}.$$
We can compute
\be \la{phim}
\phi(M)=\frac {\phi(m^+)\phi(m^-)+1} {\phi(m^+)+\phi(m^-)}
\ee
which implies $\phi(M) \leq \phi(m^+)$ and hence
\be \label {psi2}
\psi(M) \leq \psi(m^+)
\ee
(whatever the value of $m^- \in\H$).
By (\ref {psi1}), this gives (\ref{Mpsi}).

For $k \geq 1$ integer, let
\be
P_{(k)}=\sum_{j=1}^k A_{2j-1}^*(x+\alpha)
A_{2j-1}(x+\alpha).
\ee
Then $P_{(k)}$ is an increasing family of
positive self-adjoint linear maps.  In
particular, $\|P_{(k)}\|$, $\frac {\det P_{(k)}} {\|P_{(k)}\|}$
and $\det P_{(k)}$ are increasing positive functions.  It is not
difficult to see that $\|P_{(k)}\|$ (and hence $\det P_{(k)}$) is also
unbounded (since $A_j\in\SL(2,\R)$ implies $\tr P_{(k)} \geq 2k$).

\begin{lemma} 

Let $\epsilon$ be such that $\det P_{(k)}=\frac {1} {4 \epsilon^2}$. 
Then
\be \la{524}
C^{-1}<\frac {\psi(m^+(E+i\epsilon))} {2\epsilon \|P_{(k)}\|}<C.
\ee

\end{lemma}

\begin{pf}

Let $(u^\beta_j)_{j \geq 0}$
satisfy \be \la{sol} A(x+j\alpha) \cdot \left (\bm
u^\beta_j\\u^\beta_{j-1} \em \right )=\left (\bm u^\beta_{j+1}\\u^\beta_j
\em \right ),\ee $$u^\beta_0 \cos \beta+u^\beta_1 \sin \beta=0,\;|u^\beta_0|^2+|u^\beta_1|^2=1.$$  For integer $L$ define
\be \la{ul}
\|u\|_L=\left (\sum_{j=1}^{L} |u_j|^2
\right )^{1/2}.
\ee
Unlike \cite{JL1} it will be sufficient for us here to deal with the ``discrete''
definition (\ref{ul}) of $\|u\|_L,$ because we only deal with bounded
potentials, see proof of Corollary \ref{psieps}.

Theorem 1.1 of \cite {JL1} can be stated as follows
(see (2.13) in \cite {JL2}).
If $$\|u^\beta\|_L \|u^{\beta+\pi/2}\|_L=\frac {1} {2 \epsilon}$$ then
\be \la{5241}
5-\sqrt {24}<|m^+_\beta(E+i\epsilon)| \frac {\|u^\beta\|_L}
{\|u^{\beta+\pi/2}\|_L}<5+\sqrt {24}.
\ee
In other words,
\be
5-\sqrt {24}<\frac {|m^+_\beta(E+i\epsilon)|} {2 \epsilon
\|u^{\beta+\pi/2}\|^2_L}<5+\sqrt {24}.
\ee

It is immediate to see that if $L=2k$ then \be \la{9*}\|u^\beta\|^2_L=
\langle P_{(k)} \left (\bm u_1^\beta\\u_0^\beta\em \right ),\left (\bm
u_1^\beta\\u_0^\beta\em \right ) \rangle \leq \|P_{(k)}\|,\ee
with equality for $\beta$ maximizing $\|u^\beta\|_L^2$.  Thus
\be \la{detinf}
\det P_{(k)}=\inf_\beta \|u^\beta\|^2_L
\|u^{\beta+\pi/2}\|^2_L,
\ee
the infimum being attained at the critical points
of $\beta \mapsto \|u^\beta\|^2_L$.  We conclude that if
$\|\det P_{(k)}\|=\frac {1}
{4 \epsilon^2}$, then for every $\beta$,
$$\frac {|m^+_\beta(E+i\epsilon)|} {2 \epsilon \|P_{(k)}\|}<5+\sqrt {24},$$
and if $\beta$ is such that $\|u^{\beta+\pi/2}\|^2_L$ is maximal then
$$\frac {|m^+_\beta(E+i\epsilon)|} {2 \epsilon \|P_{(k)}\|}>5-\sqrt {24}.$$ This together gives (\ref{524}) with $C=5+\sqrt {24}$.
\end{pf}
\begin{rem}
Replacing (\ref{5241}) with a result of \cite{kkl} we can obtain by the same argument
that if $\epsilon$ is such that $\det P_{(k)}=\frac {1} {\epsilon^2},$ 
then  \be \la{23}
2-\sqrt {3}<\frac {\psi(m^+(E+i\epsilon))} { \epsilon \|P_{(k)}\|}<2+\sqrt {3}.
\ee
We note that $\det  P_{(k)}$ is precisely the Hilbert-Schmidt norm of operator $K$ in 
\cite{kkl} at scale $L.$ The exact value of $C$ in (\ref{524}) is not important for the present argument. 
\end{rem}

\subsection{Proof of Theorem \ref {est}}\la{4.1}

We need to estimate $\|P_{(k)}\|$ and $\det
P_{(k)}=\|P_{(k)}\| \|P_{(k)}^{-1}\|^{-1}$.  More precisely, we will show
that $\|P_{(k)}\| \leq C \|P_{(k)}^{-1}\|^{-3}$, which is just enough for
our purposes.

\begin{rem}

By Lemma \ref{linear}, we have $\|P_{(k)}\| \leq C k^3$, so to estimate
$\|P_{(k)}\| \leq C \|P_{(k)}^{-1}\|^{-3}$ it would be enough to show that
$\|P_k^{-1}\| \leq C k^{-1}$.  We do not know whether the latter estimate
holds.

\end{rem}

A key point will be to compare the dynamics of $(\alpha,A)$
with the dynamics of $(\alpha,T)$ where $T$ is in some particularly simple
triangular form.  Below, the notation $a \approx b$ ($a,b>0$)
denotes $C^{-1} a \leq b \leq C a$.

\begin{lemma} \label {TX}

Let $$T(x)=\left (\bm e^{2 \pi i \theta} & t(x)\\0&e^{-2 \pi i \theta} \em
\right ),$$ where $t$ has a single non-zero Fourier coefficient,
$t(x)=\hat t_r e^{2 \pi i r x}$.  Let $X=\sum_{j=1}^k T^*_{2 j-1}
T_{2j-1}$.  Then
\be
\|X\| \approx k (1+|\hat t_r|^2
\min \{k^2,\|2\theta-r\alpha\|_{\R/\Z}^{-2}\}),
\ee
\be
\|X^{-1}\|^{-1} \approx k.
\ee

\end{lemma}

\begin{pf}

We can compute explicitly $$T_j=\left (\bm e^{2 \pi ij\theta}&t_j\\0&
e^{-2\pi ij\theta} \em \right )$$ with $$t_j(x)=\hat t_r e^{2 \pi i
(rx+(j-1)\theta)} \frac {e^{2 \pi i j \delta}-1} {e^{2 \pi i \delta}-1},$$
where $\delta=r\alpha-2\theta$.
Write $X=\left (\bm k & x_1\\ \overline x_1 & x_2\em
\right )$.  By a straightforward computation,
\be
x_1=\hat t_r e^{2 \pi i (rx-\theta)} \sum_{j=1}^k
\frac {e^{2 \pi i (2j-1) \delta}-1} {e^{2 \pi i \delta}-1}=
\frac {\hat t_r} {e^{2 \pi i \delta}-1} e^{2 \pi i
(rx-\theta)} \left (e^{2 \pi i \delta} \frac {e^{4 \pi i k \delta}-1}
{e^{4 \pi i \delta}-1}-k \right ),
\ee
\be
x_2=k+|\hat t_r|^2 \sum_{j=1}^k
\left (\frac {\sin \pi (2j-1) \delta} {\sin \pi \delta} \right )^2=k
\left (1+\frac {2 |\hat t_r|^2} {|e^{2 \pi i \delta}-1|^2}
\left (1-\frac {\sin 4 \pi
k \delta} {2 k \sin 2 \pi \delta} \right ) \right ).
\ee
From those formulas we conclude that
\be
\det X=k^2 \left (1+\frac {|\hat t_r|^2} {|e^{2 \pi i \delta-1}|^2}
\left (1-\left (\frac {\sin 2 \pi k \delta} {k \sin 2 \pi \delta} \right
)^2 \right ) \right ).
\ee

We first estimate $x_2$.  If $k \|2 \theta-r\alpha\|_{\R/\Z} \geq \frac {1} {12}$
then $$1-\frac {\sin 4 \pi k \delta} {2 k \sin 2 \pi \delta} \approx 1$$ and
$$x_2 \approx k (1+\frac {|\hat t_r|^2} {\|2 \theta-r\alpha\|_{\R/\Z}^2}).$$
If $k \|2 \theta-r\alpha\|_{\R/\Z}<\frac {1} {3}$ then
$$1-\frac {\sin 4 \pi k \delta} {2 k \sin 2 \pi \delta}
\approx (k\|2 \theta- r\alpha\|_{\R/\Z})^2,$$
and $x_2 \approx k (1+k^2 |\hat t_r|^2)$. Since $X$ is positive, we
have 
 $k x_2 \geq |x_1|^2$, and since  $x_2 \geq k$ we have $x_2\geq \max
\{k,x_1\}.$
Thus  $\|X\| \approx x_2$, and
the first estimate follows.
\comm{To estimate $x_1$ we write
$$ |x_1|=\frac {|\hat t_r|} {|e^{2 \pi i \delta}-1|}|k(\frac {\sin 2
  \pi k \delta} {k \sin 2 \pi \delta}-1)+(e^{2\pi ik\delta}-1)\frac
  {\sin 2 \pi k \delta} {k \sin 2 \pi \delta}|.$$
Thus if  $k \|2 \theta-r\alpha\|_{\R/\Z} \geq \frac {1} {6}$ we have
  $|x_1|\leq C \frac  {k|\hat t_r|} {\|2 \theta-r\alpha\|_{\R/\Z}^2}.$
If  $k \|2 \theta-r\alpha\|_{\R/\Z} \leq \frac {1} {6}$ we get
  $|x_1|\leq Ck^2|\hat t_r|.$ In either case it follows that $|x_1|\leq C|x_2|.$
Since we also have $x_2 \geq k$, we obtain  $\|X\| \approx x_2$, and
the first estimate follows.}

We now estimate $\det X$.  For $k=1$ we have $\det X=1$.  Assume that $k>1$.
If $k \|2 \theta-r\alpha\|_{\R/\Z} \geq \frac {1} {12}$
then $$1-\left (\frac {\sin 2 \pi k \delta} {k \sin 2 \pi \delta} \right )^2
\approx 1$$ and $$\det X \approx k^2 (1+\frac {|\hat t_r|^2}
{\|2 \theta-r\alpha\|_{\R/\Z}^2}).$$
If $k \|2 \theta-r\alpha\|_{\R/\Z}<\frac {1} {12}$ then
$$1-\left (\frac {\sin 2 \pi k \delta} {k \sin 2 \pi \delta} \right )^2
\approx (k \|2 \theta-r\alpha\|_{\R/\Z})^2,$$
and $$\det X \approx k^2 (1+k^2 |\hat t_r|^2).$$  This implies the
second estimate, using that
$\|X^{-1}\|^{-1}=\frac {\det X} {\|X\|}$.
\end{pf}

\begin{lemma} \label {tildeT}

Let $T$, $t$
and $X$ be as in the previous lemma.  Let $\tilde T:\R/\Z \to
\SL(2,\C)$.  Let $\tilde X=\sum_{j=1}^k \tilde T_{2j-1}^* \tilde T_{2j-1}$.
Then
\be
\|\tilde X-X\| \leq 1, \quad \text {provided that} \quad
\|\tilde T-T\|_0 \leq c k^{-2} (1+2k \|t\|_0)^{-2}.
\ee

\end{lemma}

\begin{pf}

Notice that $\|T_j\|_0 \leq 1+j \|t\|_0$, and
\be
\|\tilde T_j-T_j\|_0 \leq \sum_{s=1}^j \binom {j} {s} \|\tilde T-T\|_0^s
\max_{1 \leq i \leq j} \|T_i\|_0^{1+s}.
\ee
Thus, if $\|\tilde T-T\|_0 k^2 (1+2k\|t\|_0)^2$ is small we have
$$\|\tilde T_j-T_j\|_0 \leq c k^{-1} (1+2k\|t\|_0)^{-1},\;1 \leq j \leq
2k-1.$$  This implies
$$\|\tilde T_j^* \tilde T_j-T_j^* T_j\|_0 \leq c k^{-1},\;1 \leq j \leq
2k-1,$$ which gives the estimate.
\end{pf}

\begin{thm} \label {blabl}

Let $n=|n_j|+1<\infty$, $N=|n_{j+1}|$.  Then
\be
\frac {\|P_{(k)}\|} {\|P_{(k)}^{-1}\|^{-3}} \leq C, \quad
C n^{C}<k<c e^{c N}.
\ee

\end{thm}

\begin{pf}

Let $\Phi$, $q_1$, $q$, $q_3$, $q_4$ be as in the definition of
$(C_0,c_0,\epsilon_0)$-goodness.

Let $\Delta>n$.  Let $|r| \leq \Delta$ minimize
$\|2\theta-r\alpha\|_{\R/\Z}$.  Then $|r| \geq n-1$.
By the Diophantine condition,
\be \label {rdc}
\|2 \theta-j\alpha\|_{\R/\Z} \geq c \max \{1+|r|,|j|\}^{-C}, \quad \text {for}
\quad j \neq r \quad \text {such that} \quad |j| \leq \Delta.
\ee
Decompose $q=t+g+h$ so that $t$ has only
the Fourier mode $r$, $g$ has only the Fourier modes $j \neq r$ with $|j|
\leq \Delta$ and $h$ is the rest.  Then
\be
\Phi(x+\alpha) A(x) \Phi(x)^{-1}=T+G+H,
\ee
where $$T=\left (\bm e^{2 \pi i \theta}&t\\0&e^{-2 \pi i \theta} \em
\right ),\;G=\left (\bm 0&g\\0&0 \em \right ),\;H=\left (\bm
q_1&h\\q_3&q_4\em \right ).$$  By (\ref {q1}) and (\ref {q}),
\be \label {estH}
\|H\|_0 \leq C e^{-c n^{-C} \Delta}+C e^{-c N}.
\ee

Let $Y=\left (\bm 1&y\\0&1 \em \right )$ be such that
$$Y(x+\alpha)(T+G)(x)Y(x)^{-1}=T(x).$$ Then we have 
\be
\hat y_j=-\hat q_j \frac {e^{-2\pi i\theta}}{1-e^{-2\pi
    i(2\theta-j\alpha)}}, \quad \text {for}
\quad j \neq r \quad \text {such that} \quad |j| \leq \Delta,
\ee
and $\hat y_j=0$ if $|j|>\Delta$ or $j=r$.  Thus, by (\ref {q}) and (\ref
{rdc}),
\be \label {estY}
\|Y-\id\|_0=\|y\|_0 \leq \sum_{j \leq C (1+|r|)^C} |\hat y_j|+
\sum_{C (1+|r|)^C<j \leq \Delta}|\hat y_j| \leq C e^{-c n(\ln
  (1+n))^{-C}} (1+|r|)^C.
\ee

Let $\Psi=Y\Phi$.  By $(C_0,c_0,\epsilon_0)$-goodness, $\|\Phi\|_0 \leq C n^C$,
which together with (\ref {estY}) gives
\be \label {psir}
\|\Psi\|_0 \leq C (n^C+e^{-c n(\ln (1+n))^{-C}} (1+|r|)^C).
\ee
Let $k_\Delta \geq 1$ be maximal
such that for $1 \leq k<k_\Delta$, if we let
\be \la {ttt}
\tilde T(x)=\Psi(x+\alpha)A(x)\Psi(x)^{-1} \quad \text {and} \quad \tilde
X=\sum_{j=1}^k \tilde T_{2j-1}^* \tilde T_{2j-1}
\ee
then
\be \label{XX}
\|\tilde X-X\|_0 \leq 1,
\ee
where $X$ is as in Lemma \ref{TX}.
Notice that $$\tilde T-T= Y(x+\alpha) H(x) Y(x)^{-1},$$ so
$$\|\tilde T-T\|_0 \leq \|Y\|_0^2 \|H\|_0.$$  By Lemma \ref {tildeT},
\eqref{XX} implies
\be
\|Y\|_0^2 \|H\|_0 \geq c k^{-2}_\Delta (1+2k_\Delta |\hat q_r|)^{-2} \geq
c k_\Delta^{-4}
\ee
(since (\ref {q}) implies $|\hat q_r| \leq C$), so that (\ref {estH}) and (\ref
{estY}) imply
\be \label {kd}
k_\Delta \geq c \min \{e^{c n^{-C} \Delta},\Delta^{-C} e^{cN}\}.
\ee

Notice that $$\|P_{(k)}(x)\| \leq \|\Psi\|_0^4 \|\tilde X(x+\alpha)\|$$ and
$$\|P_{(k)}^{-1}\|^{-1} \geq \|\Psi\|_0^{-4} \|\tilde
X(x+\alpha)^{-1}\|^{-1}.$$  Since
$\|\tilde X\| \leq \|X\|+1$ and $\|\tilde X^{-1}\|^{-1} \geq
\|X^{-1}\|^{-1}-1$ for $1 \leq k<k_\Delta$,
Lemma \ref {TX} and (\ref {psir}) imply
\be
\|P_{(k)}\| \leq C (n^C+e^{-c n (\ln (1+n))^{-C}} (1+|r|)^C)
k (1+|\hat q_r|^2 k^2), \quad C \leq k<k_\Delta,
\ee
\be
\|P_{(k)}^{-1}\|^{-1} \geq c (n^C+e^{-c n (\ln (1+n))^{-C}} (1+|r|)^C)^{-1}
k, \quad C \leq k<k_\Delta.
\ee
Thus
\begin{align} \label {bla1}
\frac {\|P_{(k)}\|} {\|P_{(k)}^{-1}\|^{-3}} \leq
&C^{(2)} \frac {n^{C^{(2)}}+e^{-c^{(3)} n (\ln (1+n))^{-C^{(2)}}}
(1+|r|)^{C^{(2)}}} {k^2}\\
\nonumber
&+C^{(2)} (n^{C^{(2)}}+e^{-c^{(3)} n (\ln (1+n))^{-C^{(2)}}}
(1+|r|)^{C^{(2)}}) |\hat q_r|^2, \quad \text {for} \quad
C^{(2)}<k<k_\Delta.
\end{align}
By (\ref {q}),
\be \label {bla2}
C^{(2)} (n^{C^{(2)}}+e^{-c^{(3)} n (\ln (1+n))^{-C^{(2)}}}
(1+|r|)^{C^{(2)}}) |\hat q_r|^2 \leq C.
\ee
Let
\be \label {bla3}
k_\Delta^-=(n^{C^{(2)}}+e^{-c^{(3)} n (\ln (1+n))^{-C^{(2)}}}
(1+\Delta)^{C^{(2)}})^{1/2}.
\ee
Then (\ref {bla1}) and (\ref {bla2}) imply
\be
\frac {\|P_{(k)}\|} {\|P_{(k)}^{-1}\|^{-3}} \leq C, \quad
k_\Delta^-<k<k_\Delta.
\ee

In order to conclude, we have to show that for $C n^C<k<c e^{c N}$ there
exists $\Delta>n$ such that $k_\Delta^-<k<k_\Delta$.  But this is clear from
(\ref {kd}) and (\ref {bla3}) that one can find such $\Delta$ with
$\Delta\approx n^C\ln k.$
\end{pf}

\begin{cor}

For $k \geq 1$, we have
$\|P_{(k)}\| \leq C \|(P_{(k)})^{-1}\|^{-3}$.

\end{cor}

\begin{pf}

It follows from Theorem \ref {blabl} and Lemma \ref {following resonance}.
\end{pf}

Set $\eps_k=\sqrt{\frac 1{4\det
P_{(k)}}}.$

\begin{cor}\label{psieps}

We have $\psi(m^+(E+i\epsilon_k)) \leq C \epsilon_k^{-1/2}$

\end{cor}

\begin{pf}
 We have  $\|P_{(k)}\|=\det
P_{(k)}\|(P_{(k)})^{-1}\|<\frac
C{\epsilon_k^2}\|P_{(k)}\|^{-1/3}.$
Thus $\|P_{(k)}\|\leq C\epsilon_k^{-3/2}$ and the statement
follows from (\ref{524}).\end{pf}

{\it Proof of Theorem \ref {est}.}
By (\ref {imm}) it is enough to show that
\be \label {imb}
c \epsilon^{1/2} \leq \Im M(E+i\epsilon) \leq C \epsilon^{-1/2}.
\ee

For any bounded potential and any solution $u$ we have
$\|u\|_{L+1}\leq C \|u\|_L,$ thus by (\ref{detinf}), $\epsilon_{k+1}>c\epsilon_k.$
Since $\frac 1\epsilon \Im
M(E+i\epsilon)$ is monotonic in
$\epsilon$ 
it therefore suffices to prove (\ref {imb})
 for $\epsilon=\epsilon_k.$ But it follows immediately from
Corollary \ref{psieps} and \eqref{psi1},\eqref{psi2}.
\qed

\begin{rem}

Corollary \ref {psieps} can be refined: for any
$0<\epsilon<1$ we have $\psi(m^+(E+i\epsilon)) \leq C \epsilon^{-1/2}$ (thus
it is not necessary to restrict to a subsequence of the $\epsilon_k$).  Indeed, by the definition of $\psi(z)$
it is enough to show that $|m^+_\beta| \leq C \epsilon^{1/2}$ for arbitrary
$\beta$.  Our proof can be easily adapted to show
$1/2$-H\"older continuity of spectral measures
$\mu^+_\beta$ associated to half-line problems
(with appropriate boundary conditions)
whose Borel transform is $m_\beta^+$, so that
\begin{align}
|m^+_\beta(E+i\epsilon)| &\leq\int\frac
{d\mu^+_\beta}{\sqrt{(x-E)^2+\epsilon^2}}\\
\nonumber
&=\int_{0}^{\frac 1{\epsilon}}dt\;\mu\left(\left\{x\in\R\,\left|
\;|x-E|<\sqrt{\frac 1{t^2} -\epsilon^2}
\right.\right\}\right)\\
\nonumber
&<C\epsilon^{-1/2}\int_0^\infty dx\;\frac{x^{1/4}}{(x+1)^{3/2}}.
\end{align}

\end{rem}

\begin{rem}

It would be interesting to obtain estimates on the modulus of absolute
continuity of the spectral measures.  It does not seem unreasonable
that for all $X,\,$
$\mu(X) \leq C |X|^{1/2}.$ 
Heuristically, the densities of the spectral measures
are unbounded just because of the presence of those countably many
(but quickly decaying) square-root singularities located at the
gap boundaries.  We point out that this is extremely similar
to what is expected from the densities of physical measures of
typical chaotic unimodal maps.
\end{rem}

\end{document}